\documentclass[a4paper,11pt,leqno]{amsart}
\usepackage{amsmath}
\usepackage{amsthm}
\usepackage{amsfonts}
\usepackage{amssymb}
\usepackage{texdraw}
\usepackage[mathscr]{eucal}
\usepackage{eucal}

\usepackage{graphicx}
\usepackage{psfrag}

\usepackage[all]{xy}

\newtheorem{theorem}{Theorem}[section]{\bf}{\it}
{\bf}{\it}
\newtheorem{proposition}[theorem]{Proposition}{\bf}{\it}
{\bf}{\it}
{\bf}{\it} 
{\bf}{\it}

{\bf}{\it}

{\bf}{\it}

{\bf}{\it}
{\bf}{\it}

\theoremstyle{definition}
{\bf}{\it}

\newtheorem{remark}[theorem]{Remark}{\bf}{\it}          

\numberwithin{equation}{section}

\newcommand{\R}{\mathbb R}

\newcommand{\N}{\mathbb N}
\newcommand{\Z}{\mathbb Z}

\newdimen\vintkern\vintkern11pt
\def\vint{-\kern-\vintkern\int}

\newcommand{\dt}{\;\mathrm{d}t}
\newcommand{\ds}{\;\mathrm{d}s}

\newcommand{\dx}{\;\mathrm{d}x}

\newcommand{\bS}{\mathbb{S}}

\newcommand{\haus}{\mathcal{H}}

\newcommand{\cC}{\mathcal{C}}
\newcommand{\cW}{\mathcal{W}}

\newcommand{\cE}{\mathcal{E}}

\newcommand{\fI}{\mathfrak{I}} 

\newcommand{\bB}{\mathbb{B}}

\newcommand{\Mod}{\mathrm{Mod}}
\newcommand{\interior}{\mathrm{int}}
\newcommand{\wind}{\mathrm{circ}}

\newcommand{\longi}{\Sigma} 

\newcommand{\Bd}{\mathrm{Bd}}

\newcommand{\Wh}{\mathrm{Wh}}

\begin{document}

\title{Quasiregular ellipticity of open and generalized manifolds}
\date{\today}
\author{Pekka Pankka}
\address{University of Jyv\"askyl\"a, Department of Mathematics and Statistics (P.O. Box 35), FI-40014 University of Jyv\"askyl\"a, Finland}\email{pekka.j.pankka@jyu.fi}
\thanks{P.P. and K.R. are supported by the Academy of Finland. J-M.W. is partially supported by the National Science Foundation grant DMS-1001669.}

\author{Kai Rajala}
\address{University of Jyv\"askyl\"a, Department of Mathematics and Statistics (P.O. Box 35), FI-40014 University of Jyv\"askyl\"a, Finland}
\email{kai.i.rajala@jyu.fi}

\author{Jang-Mei Wu}

\address{Department of Mathematics, University of Illinois,  1409 West Green Street, Urbana, IL 61822, USA}
\email{wu@math.uiuc.edu}

\subjclass[2010]{Primary 30C65; Secondary 30L10}
\keywords{quasiregular mappings, quasiregular ellipticity, decomposition spaces, Semmes metrics}

\begin{abstract}
We study the existence of geometrically controlled branched covering maps from $\R^3$ to open $3$-manifolds or to decomposition spaces $\bS^3/G$, and from $\bS^3/G$ to $\bS^3$.
\end{abstract}

\maketitle

\centerline{\emph{Dedicated to the memory of Fred Gehring,}}

\centerline{\emph{whose work has been our inspiration}}

\section{Introduction}

In this article, we study the existence of geometrically controlled branched covering maps from $\R^3$ to open $3$-manifolds or to decomposition spaces $\bS^3/G$, and from $\bS^3/G$ to $\bS^3$.
The spaces considered are associated to classical examples in geometric topology and are equipped  with a geometrically natural metric; the covering maps are discrete, open and either of bounded length distortion or quasiregular. We discuss the possibility of  extending  results in \cite{HW} and \cite{PW} on quasisymmetric parametrization to quasiregular ellipticity, and of extending the theorems in \cite{PR} on quasiregular ellipticity of link complements to more general manifolds.

To set the stage for more general discussion, we consider the classical Whitehead continuum $\Wh\subset \R^3$.

The unit sphere $\bS^3 \subset \R^4$ is a union of the closed solid $3$-tori $T=\bar B^2\times \bS^1$ and $T'=\bS^1 \times \bar B^2$. Let $S=\{0\}\times \bS^1$, $S' = \bS^1\times \{0\}$, and $L = S' \cup C$ be a smooth Whitehead link in $\bS^3$ where $C$ is an unknot in $\bS^3$ which is knotted but contractible in $T$; see Figure \ref{fig:Whitehead}.

\begin{figure}[h!]
\label{fig:Whitehead}
\includegraphics[scale=0.50]{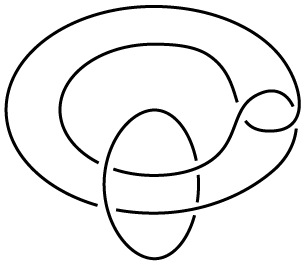}
\caption{Whitehead link $L$.}
\end{figure}

Let $\varphi \colon T \to T$ be a diffeomorphic embedding so that $\varphi(S) = C$. Denote $T_0=T$ and $T_1 = \varphi(T)$. The tori $T_k=\varphi^{k}(T)$ obtained by iterating the map $\varphi$ form a nested sequence and the continuum $\Wh = \bigcap_{k\ge 0} T_k$ is called the \emph{Whitehead continuum}.

\begin{figure}[h!]
\label{fig:Whitehead2}
\includegraphics[scale=0.50]{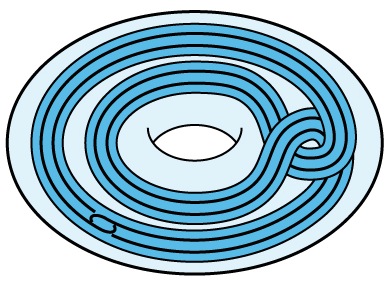}
\caption{Tori $T_0$, $T_1$, and $T_2$.}
\end{figure}

Whitehead discovered the continuum $\Wh$ in 1935 \cite{whitehead} and showed that  its complement \emph{$W=\bS^3\setminus \Wh$ is a contractible open $3$-manifold which is not homeomorphic to $\R^3$}; the Whitehead manifold $W$ is the first such manifold found. For similar reasons, the quotient space $\R^3/\Wh$ is not homeomorphic to $\R^3$. On the other hand, $(\R^3/\Wh)\times \R$ is homeomorphic to $\R^4$ showing that $\R^4$ has exotic factors; the spaces $\R^2$ and $\R^3$ however are known to have no exotic factor.

Since $W$ and $\R^3/\Wh$ are not homeomorphic to $\R^3$, the main question from the quasiconformal point of view is whether $W$ and $\R^3/\Wh$ admit metrics for which there exists a quasiregular mapping
\begin{enumerate}
\item from $\R^3$ to $W$, or \label{item:R3_W}
\item from $\R^3$ to $\R^3/\Wh$, or \label{item:R3_Wh}
\item from $\R^3/\Wh$ to $\R^3$.\label{item:Wh_R3}
\end{enumerate}

In question \eqref{item:R3_W}, it is natural to consider Riemannian metrics on $W$. Let $\pi\colon \R^3\to \R^3/\Wh$ be the quotient map.
Since the target in \eqref{item:R3_Wh} contains a non-manifold point $[\Wh]=\pi(\Wh)$, we define quasiregular mappings as in \cite{OR}. This definition, however, does not extend to the case \eqref{item:Wh_R3} when the source is not a manifold. Intriguingly, for mappings of bounded length distortion (BLD) the question \eqref{item:Wh_R3} is well-defined and has been answered by Heinonen and Rickman \cite{Heinonen-Rickman-Duke}.

Recall that a mapping $f\colon X\to Y$ between metric spaces is \emph{$L$-BLD} ($L\ge 1$) if $f$ is a discrete and open map which satisfies
\[
\ell(\gamma)/L \le \ell(f\circ \gamma) \le L\ell(\gamma)
\]
for all paths $\gamma$ in $X$, where $\ell(\cdot)$ is the length of a path.

In this article we consider questions \eqref{item:R3_W}--\eqref{item:Wh_R3} in more general context of decomposition spaces associated to defining sequences. For the open Whitehead manifold $W$ and the Whitehead (decomposition) space $\bS^3/\Wh$ we have the following results.

\begin{theorem}
\label{thm:R3_W}
There is no BLD-mapping $\R^3 \to (W,g)$ for any Riemannian metric $g$.
\end{theorem}

\medskip

A natural framework for questions \eqref{item:R3_Wh} and \eqref{item:Wh_R3} is provided by the self-similar decomposition spaces $\bS^3/G$ (or $\R^3/G$) arising from Semmes' initial packages.  Semmes showed that these spaces admit natural metrics $d$ which promote topological self-similarity in $\bS^3/G$ to geometric self-similarity in $(\bS^3/G, d)$; see Section \ref{sec:Decomposition_spaces}.
In these terminologies, $\R^3/\Wh$  and $\bS^3/\Wh$ are given by the package $(T,\varphi(T),\varphi)$.
Semmes metric $d$ on the Whitehead space $\bS^3/\Wh$ extends the topological self-similarity of the pairs $(T_k,\varphi^k)$ to the  quasi-self-similarity of the space $(\bS^3/\Wh, d)$. More precisely, there exist $0< \lambda < 1$ and $L \ge 1$ so that
\[
\lambda^k d(x,y)/L \le d(\varphi^k(x),\varphi^k(y)) \le L \lambda^k d(x,y)
\]
for all $x,y\in T$ and $k\in \N$.
\medskip

We show in Theorem \ref{thm:omit} that every non-constant quasiregular mapping $\R^3\to \bS^3/\Wh$ omits the point $[\Wh]$. On the other hand, $\bS^3/\Wh$ is compact and quasi-convex in the Semmes metric. The combination of these two facts prevents the existence of BLD-mappings $\R^3\to (\bS^3/\Wh,d)$.

\begin{theorem}\label{thm:Whitehead non-BLD ellip}
There are no BLD-mappings $\R^3 \to (\bS^3/\Wh, d)$ for any Semmes metric $d$ on $\bS^3/\Wh$.
\end{theorem}

The proof of Theorem \ref{thm:omit} is a modification of the argument in \cite{HW} and is in the spirit of Fred Gehring's linking theorem in \cite{Gehring}; see also \cite{SemmesS:Goomsw} and \cite{PW}.

We do not know whether $(\bS^3/\Wh, d)$ is quasiregularly elliptic, i.e., whether there exists a non-constant quasiregular mapping $\R^3\to (\bS^3/\Wh,d)$. However, the method of proving Theorem \ref{thm:omit} shows that many decomposition spaces, for example the space $\bS^3/\Bd$ associated to the Bing double $\Bd$, are not quasiregularly elliptic.

\begin{theorem}\label{thm:bing non-ellipitc}
Let $\bS^3/\Bd$ be the decomposition space associated to the Bing double  $\Bd$ and $d_\lambda$ a Semmes metric on $\bS^3/\Bd$ for a scaling $\lambda \in (0,1)$. Then every quasiregular map $\R^3\to (\bS^3/\Bd,d_\lambda)$ is constant.
\end{theorem}

We would like to recall that $\bS^3/\Bd$ is topologically a $3$-sphere. In this sense, the obstruction to quasiregular ellipticity is a metric phenomenon related to the properties of the decomposition $\Bd$. It is interesting to note that there are non-trivial decomposition spaces which are quasiregularly elliptic. For instance, the decomposition spaces associated to Antoine's necklaces constructed using a large number of tori are  quasisymmetric to $\bS^3$ when equipped with  Semmes metrics; see \cite[Section 17]{PW} and \cite{SemmesS:Nonblp}. However, it should be emphasized that quasisymmetric parametrizability or quasiregular ellipticity of many classical decomposition spaces, such as the space associated to  Antoine's necklace constructed using three tori, is still an open question.

The positive answer of Heinonen and Rickman's to \eqref{item:Wh_R3} follows from the Berstein--Edmonds extension theorem for PL branched covers and Rickman's sheet construction:
\emph{whenever $\bS^3/G$ is a decomposition space associated to an initial package and $d$ is a Semmes metric, there exists $n_0\ge 1$ so that for each integer  $n = 3 i \ge n_0$, there is a BLD-mapping $(\bS^3/G, d) \to \bS^3$ of degree $n$}; see \cite{Heinonen-Rickman-Duke}.

We reduce the question on the existence of a BLD-map $(\bS^3/G,d)\to \bS^3$ to the Berstein--Edmonds extension theorem by a different method and prove the following result.

\begin{theorem}
\label{thm:BLD_degree}
Let $\bS^3/G$ be a decomposition space associated to an initial package and $d$ a Semmes metric on $\bS^3/G$. Then for each $n\ge 3$ there exists a BLD-mapping $(\bS^3/G, d) \to \bS^3$ of degree $n$.
\end{theorem}

The restriction for the degree in this theorem is essential and stems from the non-existence of degree $2$ branched covers $\bS^1\times \bS^1 \times \bS^1 \to \bS^3$; see Fox \cite{Fox-1972-branched} or Berstein--Edmonds \cite{Berstein-Edmonds-1978} for a general cohomological argument.

This paper is organized as follows. In Section \ref{sec:open_Whitehead}, Theorem \ref{thm:R3_W} is proved. We prove an extension of the Berstein--Edmonds theorem (Theorem \ref{thm:branch-extension}) in Section \ref{sec:Berstein-Edmonds} and prove an extension of the Heinonen--Rickman theorem (Theorem \ref{thm:BLD_degree}) in Section \ref{sec:Decomposition_spaces}. In Section \ref{sec:Decomposition_spaces} we also discuss construction of self-similar decomposition spaces and Semmes metrics. Finally, in Section \ref{sec:QR_non-ellipticity} we discuss the topological properties of self-similar decomposition spaces which yield quasiregular non-ellipticity results and prove Theorems \ref{thm:Whitehead non-BLD ellip} and \ref{thm:bing non-ellipitc}.

\section{Non BLD-ellipticity of the open Whitehead manifold}
\label{sec:open_Whitehead}

We use the notation from the introduction. Let $T=\bar B^2\times \bS^1$ and $T' = \bS^1 \times \bar B^2$ be solid $3$-tori in $\bS^3\subset \R^4$, with disjoint interiors, for which $T\cup T' = \bS^3$. We denote $S=\{0\}\times \bS^1$, $S' = \bS^1\times \{0\}$, and by $C\subset \interior T$ the curve forming the Whitehead link $L = S' \cup C$ with $S'$. Finally, let $\varphi \colon T \to T$ be a diffeomorphic embedding such that $\varphi(S) = C$ and $\varphi(T)\subset \interior T$. Denote $T_0=T$, $T_k=\varphi^{k}(T)$ for $k \ge 1$, and  $\Wh = \bigcap_{k\ge 0} T_k$.

Let $\cC = T_0 \setminus \Wh$. Then $\varphi\colon \cC \to \cC$ is an embedding and defines an equivalence relation $\sim$ on $\cC$ by formula $x\sim y$ if and only if exists $m,k\ge 0$ for which $\varphi^m(x)=\varphi^k(y)$. With respect to this equivalence relation it is easy to see that $\cC/\!\!\!\sim$ is a closed oriented smooth $n$-manifold. We denote $\cW=\cC/\!\!\!\sim$, and by $\pi \colon \cC \to \cW$ the induced quotient map.

We also observe that there exists an infinite cyclic cover $\cW_\infty$ of $\cW$ and an embedding $\iota \colon \cC \to \cW_\infty$ for which
the following diagram
\begin{equation}
\label{diag:iota}
\xymatrix{
\cC \ar[r]^\iota \ar[dr]_{\pi} & \cW_\infty \ar[d]^{\pi_\infty}\\
& \cW }
\end{equation}
commutes, where $\pi_\infty \colon \cW_\infty \to \cW$ is a covering map. We observe also that there exists a unique homeomorphism $\psi \colon \cW_\infty \to \cW_\infty$ 
\[
\xymatrix{
\cC \ar[d]_\varphi \ar[r]^\iota & \cW_\infty \ar[d]^\psi \\
\cC \ar[r]^\iota & \cW_\infty }
\]
satisfying $\psi \circ \iota = \iota \circ \varphi $.

\begin{proposition}
\label{prop:cW}
The fundamental groups $\pi_1(\cW)$ and $\pi_1(\cW_\infty)$ contain free groups of all finite ranks.
\end{proposition}
\begin{proof}
Since $\cW$ is covered by $\cW_\infty$ it suffices to show the claim for $\pi_1(\cW_\infty)$.

Let $D=T\setminus \interior \varphi(T)$. Then $D$ is homotopy equivalent to the complement of the link $S'\cup \varphi(S)$ and $\pi_1(D)$ contains a free group. Note also that homomorphisms $\pi_1(\partial T) \to \pi_1(\partial D)$ and $\pi_1(\partial \varphi(T)) \to \pi_1(\partial D)$ induced by inclusions are one-to-one; see e.g.\;\cite[Section 16.4]{Geoghegan-book}. In addition, since $\pi_1(\partial T)$ is free abelian of rank $2$, there exists (infinitely many!) elements of $\pi_1(D)$ which are not in the image of either homomorphism.

Since $\cW_\infty = \bigcup_{k\in \Z} D_k$, where $D_k=\varphi^k(D)$, we have, by van Kampen's theorem, that the fundamental group of $\cW_\infty$ is the infinite product
\[
\cdots *_{\Z^2} \pi_1(D_{-2}) *_{\Z^2} \pi_1(D_{-1}) *_{\Z^2} \pi_1(D_0) *_{\Z^2 } \pi_1(D_1) *_{\Z^2} \pi_1(D_2) *_{\Z^2} \cdots
\]
which contains free groups of all finite ranks.
\end{proof}

By Proposition \ref{prop:cW}, $\pi_1(\cW)$ has exponential growth. Thus, by Varopoulos' theorem, $\cW$ and none of its covers is quasiregularly elliptic. However, the Whitehead manifold $\bS^3\setminus \Wh$ is not one of these covers; indeed, the Whitehead manifold is not a covering space of any closed manifold see \cite[Theorem 16.4.13]{Geoghegan-book}. The methods in the proof of Varopoulos' theorem can however be used to proof Theorem \ref{thm:R3_W}.

\begin{proof}[Proof of Theorem \ref{thm:R3_W}]
Suppose there exists a $L$-BLD map $g\colon \R^3\to (\bS^3\setminus \Wh,g)$. We observe immediately, by the path-lifting property of discrete open maps, that $(\bS^3\setminus \Wh,g)$ is complete. Thus $(\bS^3\setminus \Wh,g)$ is unbounded. Indeed, $(\bS^3\setminus \Wh,g)$ is one ended.

Since $(\bS^3\setminus \Wh,g)$ is unbounded and $f$ is BLD, we may fix a sequence $(x_k)$ in $\R^3$ for which the image of balls $fB^3(x_k,k)\subset T_1 \subset \cC$.

We fix a Riemannian metric $g_\infty$ on $\cW_\infty$ so that the embedding $\iota \colon (\cC,g) \to (\cW_\infty,g_\infty)$ is an isometry. By postcomposing $f$ with $\iota$ and precomposing with translations $x\mapsto x+x_k$ for every $k\ge 0$, we obtain $L$-BLD maps $f_k \colon B^3(0,k)\to (\cW_\infty,g_\infty)$.

Let $\widetilde{\cW_\infty}$ be the universal cover of $\cW_\infty$ and let $\widetilde{g_\infty}$ be the Riemannian metric induced by $g_\infty$, that is, $\widetilde{g_\infty}$ is the metric for which the covering map $\widetilde \pi\colon \widetilde{\cW_\infty} \to \cW_\infty$ is a local isometry. We fix also a point $x_0\in \widetilde{\pi}(f_0(0))$; note that $f_k(0)=f_0(0)$ for every $k\ge 0$.

We fix now lifts $\tilde f_k \colon B^3(0,k) \to \widetilde{\cW_\infty}$ of mappings $f_k$ so that $\tilde f_k(0)=x_0$. Then each $\tilde f_k$ is $L$-BLD and
\[
\bar B(\tilde f_k(0),k/L) \subset f\bar B^3(x_k,k)
\]
for every $k\ge 0$. By passing to a subsequence if necessary, the sequence $(\tilde f_k)$ converges locally uniformly to a (surjective) BLD-mapping $\tilde f\colon \R^n \to \widetilde{\cW_\infty}$.

Following now the proof of \cite[Theorem 1.3]{PR} almost verbatim, we observe that, by Proposition \ref{prop:cW},  $\widetilde{\cW_\infty}$ is conformally hyperbolic and the mapping $\tilde f$ is constant. This is a contradiction since $\tilde f$ is BLD and hence non-constant.
\end{proof}

\begin{remark}
By the proof of Theorem \ref{thm:R3_W} we observe that there are no proper quasiregular mappings $\R^3 \to (\bS^3\setminus \Wh,g)$ for any metric $g$.
\end{remark}

\section{Extension of branched covering maps}
\label{sec:Berstein-Edmonds}

In this section we prove a generalization of the Berstein--Edmonds extension theorem for $2$-dimensional branched covers.

\begin{theorem}\label{thm:branch-extension}
Let $W$ be a connected, compact, oriented PL $3$-manifold whose boundary consists of $p\ge 2$ components $M_0,...,M_{p-1}$ with the induced orientation.
Let $W'=N\setminus \bigcup_{j=0}^{j=p-1} \interior B_j$ be an oriented PL $3$-sphere $N$ in $\R^4$ with $p$ disjoint closed polyhedral $3$-balls $B_j$ removed, and have the induced orientation on its boundary.
Suppose that $n\ge 3$  and
$\varphi_j\colon M_j\to \partial B_j$ is a sense-preserving PL-branched cover of degree $n$, for each $j= 0,1,...,p-1$. Then there exists a sense-preserving PL-branched cover $\varphi \colon W \to W'$ of degree $n$ that extends the $\varphi_j$'s.
\end{theorem}

The so-called Berstein--Edmonds extension theorem \cite[Theorem 6.2]{Berstein-Edmonds-1979} is the case $p=2$ of Theorem \ref{thm:branch-extension}.  When $p\ge 2$ and the degree $n =3i \ge n_0$, the theorem is due to Heinonen and Rickman \cite{Heinonen-Rickman-Duke}. Hirsch \cite{Hirsch-1977} has constructed $3$-fold branched covers from $W$ to $W'$ for the types of manifolds considered in Theorem \ref{thm:branch-extension}; however his method may not be used to extend every pre-assigned boundary branched cover $\partial W \to \partial W'$.

\begin{proof}[Proof of Theorem \ref{thm:branch-extension}]
Assume $p\ge 3$. For each $j=1,\ldots, p-1$, fix two disjoint closed $2$-disks $E_j$ and $ E_j'$ on $\partial B_j$ which contain no branch values. Therefore, each of $\varphi_j^{-1}E_j$ and $\varphi_j^{-1}E_j'$ consists of $n$ mutually disjoint $2$-disks in $M_j$; we label these disks by $D_{j,i}, \, i=1,\ldots, n, $ and $D_{j,i}',\,  i=1,\ldots, n,$ respectively.

Fix a collection of mutually disjoint PL closed $3$-dim cylindrical tubes $\{U_{j,i}\colon j=1,\ldots, p-2 \,\,\text{and}\,\, i=1,\ldots, n \}$ in $W$ such  that $U_{j,i}$ connects $D_{j,i}$ to $D_{j+1,i}'$, hence  $M_j$ to $M_{j+1}$. Similarly, choose a collection of mutually disjoint PL closed $3$-dim cylindrical tubes $\{V_j\colon j=1,\ldots, p-2\}$ in $W'$ such  that 
$V_j$ connects $E_j$ to $E_{j+1}'$,  hence $\partial B_j$ to $\partial B_{j+1}$.

Choose PL homeomorphisms $\alpha_{j,i} \colon B^2(0,1)\times [0,1] \to U_{j,i}\,$ for $ j=1,\ldots, p-2 \,\,\text{and}\,\, i=1,\ldots, n, $
and PL homeomorphisms $\beta_j \colon B^2(0,1)\times [0,1] \to V_j\,$ for $ j=1,\ldots, p-2, $ with the following properties:
\begin{enumerate}

\item[(i)] $\alpha_{j,i}(B^2(0,1)\times\{0\})=D_{j,i}$, $\,\,\alpha_{j,i}(B^2(0,1)\times\{1\})=D_{j+1,i}'$,
and $\,\,\alpha_{j,i}(B^2(0,1)\times (0,1)) \subset \interior W$;

\item[(ii)]  $\beta_j(B^2(0,1)\times\{0\})=E_j$, $\,\,\beta_j(B^2(0,1)\times\{1\})=E_{j+1}'\,$,
and $\,\beta_j(B^2(0,1)\times (0,1)) \subset \interior W'$;  and

\item[(iii)] $\beta_j \circ \alpha_{j,i}^{-1}|D_{j,i} = \varphi_j|D_{j,i}\,$ and
$\,\beta_{j+1} \circ \alpha_{j+1,i}^{-1}|D'_{j+1,i} = \varphi_{j+1}|D'_{j+1,i}$.
\end{enumerate}

Extend $\varphi_j$'s to a map $\Phi$ from $Q :=  \partial W \cup \bigcup_{j=1}^{ p-2} (\bigcup_{i=1}^{n} U_{j,i})$ to $Q' := \partial W' \cup \bigcup_{j=1}^{ p-2} V_j$ as follows: $\Phi|M_j=\varphi_j$ for each $j=0,\ldots, p-1$, and

\[
\Phi|(U_{j,i}\setminus \partial W) =\beta_j \circ \alpha_{j,i}^{-1}|(U_{j,i}\setminus \partial W)
\]
on every tube $U_{j,i}$.

The map $\Phi$ is well-defined in view of (iii). It is clear that $\Phi$ is an $n$-fold branched cover from $Q$ to $Q'$ and from each of the two boundary components of $W\setminus Q$ to the corresponding boundary component of $W' \setminus Q'$. Since $\overline{W'\setminus Q'}$ is homeomorphic to $\bS^2 \times [0,1]$ and $\partial (W'\setminus Q')$ consists of two $2$-spheres,  $\Phi$ can be extended to an $n$-fold branched cover from $W$ to $W'$ by the  Berstein--Edmonds theorem.

\end{proof}

\begin{remark}
\label{rmk:B-E_rmk}
Theorem \ref{thm:branch-extension} is false when degree is $2$. By a theorem of Fox \cite{Fox-1972-branched}, there is no $2$-fold branched covering from $\bS^1\times \bS^1 \times \bS^1$ to $\bS^3$; see also Berstein--Edmonds \cite{Berstein-Edmonds-1978} and \cite{Hirsch-Neumann}. Let $W$ be $\bS^1\times \bS^1 \times \bS^1$ with two disjoint closed $3$-balls removed, so $\partial W$ consists of two $2$-spheres. Let $W'$ be $\bS^3$ with two disjoint closed $3$-balls removed, so $\partial W$ also consists of two $2$-spheres. Then any $2$-fold branched covering map  $\partial W\to \partial W'$ can not be extended to $W\to W'$.
\end{remark}

\section{Initial packages, Semmes metrics, and  the Heinonen-Rickman existence theorem}
\label{sec:Decomposition_spaces}

In this section, we discuss the decomposition spaces $\bS^3/G$ arising from sequences of cubes-with-handles and the Semmes metrics on such spaces. As in \cite[Theorem 8.17]{Heinonen-Rickman-Duke}, an application of Theorem \ref{thm:branch-extension} shows that all these decomposition spaces admit branched covering maps $\bS^3/G \to \bS^3$. When spaces are equipped  with Semmes metrics, the branched covers may be chosen to be BLD. 

Let $H\subset \bS^3$ be a smooth cube-with-handles. Recall that a compact subset $H\subset \R^3$ is a cube-with-$g$-handles if $H$ has a handle decomposition into a $0$-handle and $g$ $1$-handles for $g\ge 0$.

A tuple $\fI = (H,H_1,\ldots, H_m, \varphi_1,\ldots, \varphi_m)$ is an \emph{initial package} if $H_i$ are pair-wise disjoint
cubes-with handles in the interior of $H$ and each $\varphi_i \colon H \to H_i$ is a homeomorphism which can be extended to be diffeomorphic in some neighborhood of $H$.

The decomposition space $\bS^3/G_\fI$ associated to an initial package $\fI=(H,H_1,\ldots, H_m, \varphi_1,\ldots, \varphi_m)$ is defined as follows. Denote $L_0 = H$ and $L_k = \bigcup_{i=1}^m \varphi_i(L_{k-1})$ for every $k\ge 1$. We call $(L_k)$ the \emph{associated sequence of $\fI$}. 

Let also $L_\infty = \bigcap_{k=1}^\infty L_k$ and denote by $G_\fI$ the decomposition of $\bS^3$ whose elements are the components of $L_\infty$ and singletons in $\bS^3\setminus L_\infty$. We denote by $\pi_\fI \colon \bS^3 \to \bS^3/G_\fI$ the canonical map. We call $L_\infty$ the \emph{limit set of $(L_k)$}.

For every $x\in \pi_{\fI}(L_\infty)$, there exists a unique component $L_k(x)$ of $L_k$ so that $x\in \pi_\fI(L_k(x))$. We call the sequence $(L_k(x))$ the \emph{branch of $x$}. Note that, for every $k\ge 1$, there exists a (unique) sequence $I^x_k=(i^x_1,\ldots, i^x_k)\in \{1,\ldots, m\}^k$ so that $\varphi_{i^x_1}\circ \cdots \circ \varphi_{i^x_k} \colon H \to L_k(x)$. We denote this map by $\varphi^x_k$.


Although $\bS^3/G_{\fI}$ is metrizable, there is no canonically defined metric. We equip $\bS^3/G$ with a natural metric (a Semmes metric) introduced in \cite{SemmesS:Goomsw}; see also \cite{HW}.

Given $\lambda\in (0,\infty)$ there exists a Riemannian metric $g_{\fI,\lambda}$ on $\bS^3\setminus L_\infty$ so that the mapping $\varphi_I =\varphi_{i_1}\circ \cdots \circ \varphi_{i_k}$ restricts to a $\lambda^k$-similarity $H\setminus L_1 \to L_k\setminus L_{k+1}$ for every $I=(i_1,\ldots, i_k)\in \{1,\ldots, m\}^k$.  Recall that a mapping $f\colon X\to Y$ between metric spaces is a \emph{$\mu$-similarity} if $|f(x)-f(x')|=\mu |x-x'|$ for all $x,x'\in X$.

Note that, by compactness of $\bS^3\setminus \interior L_0$ and $L_0\setminus \interior L_1$, the space $(\bS^3\setminus L_\infty, g_{\fI,\lambda})$ is quasiconformal to $(\bS^3\setminus L_\infty,g_{\fI,1})$ for every $\lambda\in (0,\infty)$.

Let $d_{\fI,\lambda}$ be the distance function associated to the Riemannian metric $g_{\fI,\lambda}$. For $0< \lambda<1$, the metric completion of $(\bS^3\setminus L_\infty, d_{\fI,\lambda})$ is a path metric space homeomorphic to $\bS^3/G$ and the embedding $\bS^3\setminus L_\infty \to \bS^3/G$ is an isometry.

We call any metric $d_\lambda$, that is bilipschitz equivalent to $d_{\fI,\lambda}$, a \emph{Semmes metric} and the space $(\bS^3/G_{\fI},d_\lambda)$ a \emph{Semmes space}. The parameter $\lambda$ is called the \emph{scaling constant of the metric}.

Using Theorem \ref{thm:branch-extension} we formulate the Heinonen--Rickman existence theorem \cite[Theorem 8.17]{Heinonen-Rickman-Duke} as an extension theorem, and give a new proof.

\begin{theorem}
\label{thm:geometric_Heinonen-Rickman}
Let $\fI=(H,H_1,\ldots, H_m,\varphi_1,\ldots, \varphi_m)$ be an initial package and $f \colon \partial H\to \bS^2$ a BLD-map of degree $n\ge 3$. Then there exists a Semmes metric $d_\lambda$ on $\bS^3/G_{\fI}$ and a BLD-map $F \colon (\bS^3/G_{\fI}, d_\lambda) \to \bS^3$ of degree $n$ so that $F|\partial H = f$.
\end{theorem}

\begin{proof}
Let $\mathcal L=(L_k)_{k\ge 0}$ be the sequence associated to $\fI$ and $L_\infty$ the limit set of $\mathcal L$.
We may assume that $H$ is contained in the interior of the unit ball $\mathbb B^3$. Fix a number $\lambda \in (0,1)$ so that $\interior \mathbb B^3$ contains $m$ pair-wise disjoint closed balls of radius $\lambda$. We show that for a Semmes metric with  scaling $\lambda$, the space $(\bS^3/G_{\fI}, d_\lambda)$ admits a BLD-map as claimed.

Choose another initial package $\fI'=(B,B_1,\ldots, B_m,\psi_1,\ldots, \psi_m)$, with $B=\mathbb B^3$, $(B_i)_{1\le i\le m}$ pair-wise disjoint closed balls of radius $\lambda$ in  $\interior B$, and $\psi_i\colon B \to B_i$ similarity maps. Let $\mathcal L'=(L'_k)_{k\ge 0}$ be the sequence associated to $\fI'$ and  $L'_\infty$ the limit set of $\mathcal L'$.

Set $g_i = \psi_i \circ f \circ \varphi_i^{-1}|\partial H_i$ for every $1\le i \le m$, and fix a degree $n$ BLD-map $g_{0}\colon \partial B^3(0,2) \to \partial B^3(0,2)$.

By Theorem \ref{thm:branch-extension}, there exist sense preserving degree $n$ extensions $G_0 \colon L_0\setminus \interior L_1 \to  L'_0\setminus \interior L'_1$ and $G_{-1} \colon B^3(0,2)\setminus \interior L_0 \to B^3(0,2) \setminus \interior L'_0$ satisfying $G_0|\partial L_0 = f = G_{-1}|\partial L_0$, $G_0|\partial H_i = g_i$, and $G_{-1}|\partial B^3(0,2) =g_{0}$. The mapping $g_0$ may be extended to a degree $n$ BLD branched cover $G_{-2}\colon \bS^3\setminus \interior B^3(0,2)\to \bS^3\setminus \interior B^3(0,2)$, by coning with respect to a fixed point in $\bS^3\setminus B^3(0,2)$,

Having mappings $G_0$, $G_{-1}$ and $G_{-2}$ at our disposal, we may find a degree $n$ BLD branched cover $G \colon \bS^3\setminus L_\infty \to \bS^3\setminus L'_\infty$  satisfying $G \circ \varphi_I = \psi_I \circ G_0$
\[
\xymatrix{
L_0\setminus \interior L_1 \ar[r]^{G_0} \ar[d]_{\varphi_I} & L'_0\setminus \interior L'_1 \ar[d]^{\psi_I} \\
L_k \setminus \interior L_{k+1} \ar[r]^G & L'_k \setminus \interior L'_{k+1}}
\]
for every $I=(i_1,\ldots, i_k)\in \{1,\ldots, m\}^k$, where $\varphi_I=\varphi_{i_1}\circ \cdots \circ \varphi_{i_k}$ and $\psi_I$ is a $\lambda^k$-similarity.

Since $G$ induces a Lipschitz map from $(\bS^3\setminus \pi_{\fI}( L_\infty), d_\lambda)$ to $\bS^3 \setminus L'_\infty$ and $\pi_\fI(L_\infty)$ is compact and $0$-dimensional \cite[Proposition II.9.1]{DavermanR:Decm}, $G$ extends to a BLD branched covering map $F \colon \bS^3/G \to \bS^3$.
\end{proof}

Theorem \ref{thm:BLD_degree} follows from Theorem \ref{thm:geometric_Heinonen-Rickman} and the remark below.

\begin{remark}\label{rmk:exitence-branched-cover}
Let $H$ be a cube-with-handles in $\R^3$ of genus $g $. The existence of a degree $n\, (\ge 2)$ BLD-map $f \colon \partial H\to \bS^2$ follows from the Hurwitz Existence Theorem. We outline the steps as given in \cite[(2.2)]{Berstein-Edmonds-1979}.  

Fix an orientation of $\bS^2$, a set of $k=2+2g$ points $\{q_1,\ldots,q_{k}\}$ in $\bS^2$, and a base point $* \in \bS^2\setminus \{q_1,\ldots,q_{k}\} $. Then $\pi_1(\bS^2\setminus \{q_1,\ldots,q_{k}\},*)$ is a free group on $k$ generators $x_1, \ldots, x_k$ modulo the relation $x_1 x_2\ldots x_k=1$. Denote by $\mathscr S_n$ the symmetric group on $n$ letters. We fix a homomorphism
\[
\rho\colon \pi_1(\bS^2\setminus  \{q_1,\ldots,q_{k}\}, *) \to \mathscr S_n, 
\]
with the properties that   $\rho (x_1) =\rho (x_2)^{-1}$ is a cycle of length $n$ and $\rho (x_3) =\cdots= \rho(x_k)$ is a transposition. Then   $\bS^2\setminus  \{q_1,\ldots,q_{k}\}$ has a PL $n$-fold cover $\Omega$ whose fundamental group is isomorphic to $\text{ker}\,\rho $ and  
\[
\bS^2\setminus  \{q_1,\ldots,q_{k}\} \approx \Omega/ \rho(\pi_1(\bS^2\setminus  \{q_1,\ldots,q_{k}\}, *)).
\] 
Denote by $M$ the PL closed $2$-manifold which is the compactification of $\Omega$. The unbranched covering extends to a degree $n$ PL map $\varphi\colon M \to \bS^2$  whose local degrees at the branch points $p_1,\ldots,p_k$ are $\deg(\varphi, p_1)=\deg(\varphi, p_2)=n$ and $\deg(\varphi, p_3)=\cdots=\deg(\varphi, p_k)=2$. By the Riemann-Hurwitz condition,  
\[
\chi(M) = n\chi(\bS^2) -\sum_i \, (\deg (\varphi, p_i)-1).
\]
Thus the Euler characteristic $\chi(M)$ of $M$ is $2-2g$. Since $\rho$ is transitive, $M$ is connected. Furthermore, since $\bS^2$ is orientable, so are $\Omega$ and its compactification $M$. Therefore, $\partial H$ and $M$ are PL-homeomorphic and the existence of a degree $n$ BLD-map $f$ follows. 
\end{remark}

\begin{remark}
The sequence $(L_k)$ associated to an initial package $\fI$ is a (self-similar) defining sequence for the decomposition space $\bS^3/G_{\fI}$. More generally, a sequence $\mathcal{X}=(X_k)$ is a \emph{defining sequence} in $\bS^3$ if each $X_k$ is a closed set satisfying $X_{k+1} \subset \interior X_k$. We refer to Daverman \cite{DavermanR:Decm} for details on decompositions, quotient spaces associated to decompositions, and their topological applications.

The construction of Semmes metrics can be carried over on decomposition spaces associated to defining sequences given by cubes-with-handles. If, in addition, the defining sequence has \emph{finite type}, as defined in \cite[Section 4]{PW}, the associated Semmes spaces $(\bS^3/G,d_\lambda)$ admit good geometric properties. For example, the proof of Theorem \ref{thm:geometric_Heinonen-Rickman} extends to this class of spaces almost verbatim; note that sequences $\mathcal L$ and $\mathcal L'$ in the proof of Theorem \ref{thm:geometric_Heinonen-Rickman} are defining sequences of finite type.
\end{remark}


\section{Quasiregularly Non-Elliptic Decomposition Spaces}
\label{sec:QR_non-ellipticity}

Let $\fI=(H,H_1,\ldots, H_m, \varphi_1,\ldots, \varphi_m)$ be an initial package and $(L_k)$ the associated sequence as in Section \ref{sec:Decomposition_spaces}, i.e\;$L_0=H$ and $L_k = \bigcup_{i=1}^m \varphi_i(L_{k-1})$ for $k\ge 1$.

Under the assumption that components of $L_{k+1}$ are contractible in $L_k$ for every $k\ge 0$, spaces $\R^3/G_\fI$  and $\bS^3/G_\fI$ are generalized (homology) $3$-manifolds by \cite[Corollary V.1A]{DavermanR:Decm} and orientable generalized (cohomology) $3$-manifolds in the sense of Heinonen and Rickman; see \cite[Definition 1.6, Example 1.4(c), and Proposition 8.16]{Heinonen-Rickman-Duke}. In particular, these spaces admit local degree theory as discussed in \cite[Section 1]{Heinonen-Rickman-Duke}.

We assume from now on that $(\bS^3/G_{\fI},d_\lambda)$ is a Semmes space which is a generalized cohomology $3$-manifold in the sense of Heinonen and Rickman.

We define quasiregular mappings $\R^3\to \bS^3/G_{\fI}$ using the so-called \emph{metric definition} (see \cite{OR}). Let $f\colon \R^3 \to \bS^3/G_{\fI}$ be a continuous map. The \emph{(metric) distortion of $f$ at $x\in \R^3$} is
\[
H_f(x)=\limsup_{r \to 0} \frac{L(x,r)}{\ell(x,r)},
\]
where
\[
L(x,r)=\max_{y \in \overline{B}(x,r)} |f(y)-f(x)|
\quad \mathrm{and} \quad \ell(x,r)=\min_{y \in S(x,r)} |f(y)-f(x)|.
\]

A non-constant continuous map $f\colon \R^3 \to \bS^3/G_{\fI}$ is \emph{quasiregular} if $f$ is sense-preserving, discrete and open, and if there exists $H<\infty$ for which $H_f(x)\leq H$ for almost every $x \in \R^3$ and $H_f(x)< \infty$ for every $x \in \R^3$.

Since quasiregular mappings $\R^3\to \bS^3/G_{\fI}$ are branched covers, that is, discrete and open mappings, they have the path-lifting property (\cite[Section 3.3]{Heinonen-Rickman-Duke}.

As in the classical theory, we have \emph{Poletsky's inequality} for non-constant quasiregular mappings $\R^3 \to \bS^3/G_{\fI}$ (\cite[Corollary 11.2]{OR}): Let $\Gamma$ be a path family in $\R^3$ and $f\colon \R^3 \to \bS^3/G_{\fI}$ a non-constant quasiregular map. Then
\begin{equation}
\label{eq:pole}
\Mod_3(f (\Gamma)) \leq C \Mod_3(\Gamma),
\end{equation}
where $C$ depends on the distortion of $f$.

Recall that the \emph{conformal modulus} $\Mod_3(\Gamma)$ of a family $\Gamma$ in $\R^3$ is
\[
\Mod_3(\Gamma) = \inf_{\rho} \int_{\R^3} \rho(x)^3 \dx,
\]
where $\rho$ is a non-negative Borel function satisfying
\[
\int_{\gamma} \rho \ds \geq 1 \quad \text{for all locally rectifiable } \gamma \in \Gamma.
\]
The same definition, after replacing $\mathrm{d}x$ by integration over Hausdorff $3$-measure in $\bS^3/G_{\fI}$, gives the conformal modulus $\Mod_3(f\Gamma)$ of $f\Gamma$.

The rest of this section is devoted to the proof of a general non-ellipticity result (Theorem \ref{thm:QR non-ellip decomp}) which covers Theorem \ref{thm:bing non-ellipitc}.

We show that, under the assumptions of Theorem \ref{thm:QR non-ellip decomp}, the image of a non-constant quasiregular map does not contain points of $\pi_{\fI}(L_\infty)$, where $L_\infty$ is the limit of the sequence $(L_k)$ associated to $\fI$. In particular, the image is an $n$-manifold that has infinitely many ends; recall that a manifold $M$ has at least $q$ ends if there exists a compact set $E\subset M$ for which the set $\overline{M\setminus E}$ has at least $q$ non-compact components.
This contradicts the Holopainen-Rickman Picard theorem for quasiregular mappings \cite{HolopainenRickman:ricci} which states that \emph{for every $K\ge 1$ and $n\ge 2$ there exists $q=q(n,K)$ so that a manifold $M$ has at most $q$ ends if there exists a non-constant quasiregular mapping $\R^n \to M$.}


\subsection{Circulation}

For the statement of Theorem \ref{thm:QR non-ellip decomp} we define the notion of circulation. Let $\fI=(H,H_1,\ldots, H_m,\varphi_1,\ldots,\varphi_m)$ be an initial package, $(L_k)$ the associated sequence and $L_\infty$ the limit set.

A simple closed smooth curve $\alpha \colon \bS^1 \to \partial H$ is a \emph{meridian of $H$} if  $\alpha$ is not contractible on $\partial H$ but there exists a continuous map $\varphi\colon \bB^2 \to H$ so that $\varphi|\partial \bB^2 = \alpha$. We denote the collection of such maps by $\cE(H,\alpha)$.

Given $k\ge 0$, a \emph{longitude in $L_k$} is an unweighted $PL$ $1$-cycle (i.e. a sum of finitely many closed PL-curves) $\sigma$ in $L_k$ with the property that
\[
|\sigma| \cap \varphi(\bB^2)  \neq \emptyset
\]
for all meridians $\alpha$ of $H$, and all $\varphi \in \cE(H,\alpha)$. In other words, longitudes are the $1$-cycles in $L_k$ that are linked with meridians of all cubes-with-handles of $L_k$. We denote the set of longitudes by $\Sigma(L_k)$.

We say that the \emph{order of circulation of $\fI$ is at least $\omega\ge 0$} if there exists a meridian $\alpha$ of $H$ and a constant $C>0$ so that
\[
\wind(L_k,\alpha,H) := \min_{\varphi\in \cE(H,\alpha)} \min_{\sigma\in \longi(L_k)} \# (|\sigma| \cap \varphi(\bB^2)) \ge C \omega^k
\]
for every $k\ge 0$.

\begin{remark}
In many concrete examples of initial packages the order of circulation is easy to estimate. For example, for the initial packages associated to the Bing double and the Whitehead continuum has circulation at least $2$ by a lemma of Freedman and Skora (\cite[Lemma 2.4]{FreedmanM:Strags}); see \cite[Section 17]{PW} for a more detailed discussion. Note also that the definition \cite[Definition 9.2]{PW} for circulation for defining sequences of finite type and is more general than the definition we consider here.
\end{remark}

\subsection{Statements}

The following theorem is the main result of this section; recall that, by discussion in the beginning of Section \ref{sec:QR_non-ellipticity}, the quotient space $\bS^3/G_{\fI}$ is orientable under the contractibility assumption posed to the initial package.

\begin{theorem}
\label{thm:QR non-ellip decomp}
Let $\fI=(H,H_1,\ldots, H_m,\varphi_1,\ldots, \varphi_m)$, $m\geq 2$, be an initial package so that $H_i$ is contractible in $H$ for each $i=1,\ldots, m$, and let $d_\lambda$ be a Semmes metric with scaling constant $0<\lambda<1$ on $\bS^3/G_{\fI}$. Suppose that $\fI$ has order of circulation at least $\omega>m^{2/3}$. Then every quasiregular map $\R^3\to (\bS^3/G_{\fI},d_\lambda)$ is constant.
\end{theorem}

Since, for the Bing double, $m=2$ and the order of circulation is at least $2$, Theorem \ref{thm:bing non-ellipitc} follows from Theorem \ref{thm:QR non-ellip decomp}.

\begin{remark}
By modifying the method in \cite{PW}, we may generalize Theorem \ref{thm:QR non-ellip decomp} to include the class of defining sequences of finite type. Here we consider only the self-similar sequences for simplicity.
\end{remark}

As discussed above, Theorem \ref{thm:QR non-ellip decomp} is deduced by combining the Holopainen--Rickman Picard theorem for quasiregular mappings and the following result, which forces a quasiregular mapping to omit points when the circulation is large.

\begin{theorem}
\label{thm:omit}
Let $\fI=(H,H_1,\ldots, H_m,\varphi_1,\ldots, \varphi_m)$, $m\geq 1$, be an initial package so that $H_i$ is contractible in $H$ for each $i=1,\ldots, m$, and let $d_\lambda$ be a Semmes metric with scaling constant $0<\lambda<1$ on $\bS^3/G_{\fI}$. Suppose that $\fI$ has order of circulation at least $\omega>m^{2/3}$. Then $\pi_{\fI}(L_\infty) \cap f(\R^3) = \emptyset$ for all non-constant quasiregular mappings $f\colon \R^3 \to (\bS^3/G_{\fI},d_\lambda)$.
\end{theorem}

\begin{proof}[Proof of Theorem \ref{thm:QR non-ellip decomp} Assuming Theorem \ref{thm:omit}]
Let $f\colon \R^3\to (\bS^3/G_{\fI},d_\lambda)$ be a quasiregular map. We may assume that on $(\bS^3/G_{\fI})\setminus \pi(L_\infty)$ the metric $d_\lambda$ is given by a Riemannian metric $g_\lambda$. Then, by Theorem \ref{thm:omit}, $f(\R^3) \subset \bS^3/G_{\fI}\setminus \pi(L_\infty)$. Since $\pi(L_\infty)$ is a Cantor set by our assumption $m\geq 2$, $\bS^3/G_{\fI}\setminus \pi(L_\infty)$ is a Riemannian manifold with infinitely many ends. Thus, by the Holopainen--Rickman Picard theorem, $f$ is constant.
\end{proof}

\subsection{Proof of Theorem \ref{thm:omit}}
Suppose $f$ is a non-constant quasiregular map from $\R^3$ to $\bS^3/G$ having a value in $\pi(L_\infty)$, that is, there exists $q=f(p) \in \pi(L_{\infty})$.

We fix a normal neighborhood $U$ of $p$, i.e. a domain such that $U \cap f^{-1}(q)=\{p\}$ and $f(\partial U)=\partial f(U)$. Let $(L_k(q))$ be the branch of $q$. Since $L_k(q) \to \{q\}$ in the Hausdorff sense as $k\to \infty$, there exists $k_0\ge 1$ such that cubes-with-handles in the sub-branch $(L_k(q))_{k\geq k_0}$ are contained in the open set $f(U)$. By the self-similarity of $(L_k)_{k\ge 0}$, we may assume that $k_0=0$, and we have $H\subset f(U)$.

Given a longitude $\sigma \in \Sigma(L_k)$ for $k\ge 0$, we denote by $\sigma^*$ the union of subcurves of $f^{-1}(\sigma)$ which are contained in $U \cap f^{-1}(H)$. Moreover, we denote
\[
\Sigma(L_k)^* = \{\sigma^* \colon \sigma \in \Sigma(L_k)  \}.
\]
Then $f(\Sigma(L_k)^*)=\Sigma(L_k)$ and
\begin{equation}
\label{eq:poletsky}
\Mod_3 (\Sigma(L_k)) \le C \Mod_3 (\Sigma(L_k)^*),
\end{equation}
by Poletsky's inequality \eqref{eq:pole}, where $C$ depends on the distortion of $f$.

By \cite[Proposition 4.5]{HW} and \cite[Proposition 12.1]{PW}, we have the lower bound
\begin{equation}
\label{eq:down}
\Mod_3(\Sigma(L_k)) \geq C m^{-2k},
\end{equation}
where $C$ does not depend on $k$.

Thus it suffices to prove the upper estimate
\begin{equation}
\label{eq:up}
\Mod_3(\Sigma(L_k)^*) \leq C \omega^{-3k},
\end{equation}
where $C$ does not depend on $k$. Indeed, assuming \eqref{eq:up}, the theorem follows by combining \eqref{eq:poletsky} together with estimates \eqref{eq:down} and \eqref{eq:up}, and letting $k  \to \infty$.

We begin the proof of \eqref{eq:up} by fixing a meridian $\alpha$ of $H$ and a constant $C_0>0$ for which
\begin{equation}
\label{eq:kroh2}
\wind(L_k, \alpha, H) \geq C_0 \omega^k
\end{equation}
for all $k\ge 0$.

Let $W$ be a tubular neighborhood of $\partial H$ contained in $fU$ so that $W\cap L_1=\emptyset$; recall that $H=L_0$. Since $f|U$ is a proper map, $f(B_f\cap U)$ is a closed set in $fU$, where $B_f$ is the branch set of $f$. Furthermore, $f(B_f\cap U)$ has topological dimension at most $n-2$. Thus there exists a simple loop $\hat{\alpha} \colon [0,1] \to H\cap W\setminus f(B_f\cap U)$ which is homotopic to $\alpha$ in $W$.  Thus 
\[
\# (|\sigma|\cap \hat{\varphi}(\bB^2)) \ge C_0\omega^k
\]
for every longitude $\sigma \in \Sigma(L_k)$ and every $\hat{\varphi} \colon \bB^2\to H$ extending $\hat{\alpha}$.

Let $\{x_1,\ldots, x_j\}=f^{-1}(\hat{\alpha}(0))\cap U$. We denote by $\zeta \colon [0,\infty) \to H \cap W$ the infinite iterate to $\hat{\alpha}$, i.e. $\zeta(n+t)=\hat{\alpha}(t)$ for all $t\in [0,1)$ and $n\in \N$. Let $\tilde \zeta$ be a lift of $\zeta$ starting at $x_1$. Since $U$ is a normal neighborhood, we have $\tilde \zeta \colon [0,\infty) \to U$. In particular, $\tilde \zeta(n) \in \{x_1,\ldots, x_j\}$ for all $n\in \N$. By uniqueness of the lifts, there exists $1 < i \le j$ for which $\tilde \zeta(i)=x_1$ and $\tilde \zeta|[0,i]$ is a simple loop.

We may approximate $\tilde\zeta|[0,i]$ by a piece-wise affine Jordan curve $\gamma \colon \bS^1 \to \R^3$ for which $\beta= f \circ \gamma$ is contained in $H\cap W$ and homotopic to $\alpha$ in $W$. By coning $\gamma$ using a point in $\R^3$, we obtain a PL map $ \bB^2 \to \R^3$ which we denote by $\xi$. Let $T_1,\ldots, T_\ell$ be the $2$-simplices triangulating the image of $\xi$ obtained in the coning. Finally, we also choose and fix a regular neighborhood $A$ of $\beta(\bS^1)$ contained in $H\cap W$.

We extend $\xi$ to a map $\xi \colon \bB^2\times \bB^3(\delta) \to \mathbb{R}^3$ by $(x,y)\mapsto \xi(x)+y$. Then, for each $y \in \bB^{3}(\delta)$, the loop
\[
f \circ \xi|\partial \bB^2 \times \{y\}
\]
is homotopic to $\beta$ in $A$ and hence homotopic to $\alpha$ in $H\cap W$.

By \eqref{eq:kroh2},
\begin{equation}
\label{eq:lakko}
\#(|\sigma^*| \cap \xi(\bB^2\times\{y\})) \geq \# (|f\sigma^*| \cap f(\xi(\bB^2\times\{y\}))) \geq C_0 \omega^k
\end{equation}
for every $\sigma^* \in \Sigma(L_k)^*$ and $y\in \bB^3(\delta)$.

We use \eqref{eq:lakko} to show that there exists $M>0$ not depending on $k$ such that
\begin{equation}
\label{eq:lowhaus}
\haus^1(|\sigma^*|) \geq M \omega^k \quad \text{for every } \sigma^* \in \Sigma(L_k)^*.
\end{equation}

Assuming this, the theorem follows since
\[
M^{-1} \omega^{-k} \chi_{\xi(\bB^2 \times \bB^3(\delta))}
\]
is an admissible function for $\Sigma(L_k)^*$.

To prove \eqref{eq:lowhaus}, let $\sigma^* \in \Sigma(L_k)^*$, and analyse the map $\xi \colon \bB^2 \times \bB^3(\delta) \to  \R^3$. 

For every $1\le i \le \ell$, let $P_i = y_i + \operatorname{span}\{v_i^i, v_2^i\} \subset \R^3$ be the affine plane containing the triangle $T_i$. We fix $u=(u_1,u_2,u_3)\in \bS^2$ so that $v$ is not contained in any of the $2$-dimensional subspaces $\operatorname{span}\{v_i^i, v_2^i\}$ for all $1\le i\le \ell$.

Given $\sigma^* \in \Sigma(L_k)^*$, \eqref{eq:lakko} implies that for every $0\leq t \leq \delta$,
$$
\# (|\sigma^*| \cap (\cup_{i=1}^\ell(T_i + tu))) \geq C_0 \omega^k.
$$
In particular, there exists $1 \leq i(t)\leq \ell$ such that
$$
\# (|\sigma^*| \cap (T_{i(t)} + tu)) \geq C_0 \omega^k \ell ^{-1}.
$$
We conclude that there exists $i_0$ depending on $\sigma^*$, and a set $E \subset [0,\delta]$ with $\haus^1(E)\geq \delta \ell^{-1}$,
such that
\begin{equation}
\label{eq:tri}
\# (|\sigma^*| \cap (T_{i_0} + tu)) \geq C_0 \omega^k \ell ^{-1}
\end{equation}
for every $t \in E$. 

Next we note that $\sigma^*\in \Sigma(L_k)^*$ consists of finitely many loops $\sigma^*_j$. Indeed, this follows from path lifting for branched covers and the fact that $U$ is a normal domain of $f$. We may assume that each $\sigma^*_j$ is rectifiable, otherwise there is nothing to prove.

By applying an isometry, we may assume that 
\[
T_{i_0} \subset \R^2 \times\{0\}, \quad \text{and } u_3>0.
\]
Let $\Phi\colon \R^3 \to\R$ be the projection $(x_1,x_2,x_3)\mapsto x_3$ and set
\[
\tilde{\gamma}_j= \Phi \circ \sigma^*_j.
\]
Let $\gamma_j \colon [0,R_j] \to \R$ be the reparametrization of $\tilde{\gamma}_j$ by arclength. In particular,
\[
\sum_j R_j \leq \haus^1(|\sigma^*|)
\]
since $\Phi$ is a projection. Since $u_3>0$, we have by \eqref{eq:tri}, 
\[
\sum_j \#(\gamma_j^{-1}(t)) \geq C_0 \ell^{-1} \omega^k
\]
for all $t$ with $t u^{-1}_3\in E$.

Since the $\gamma_j$'s are parametrized by arc length, we get, by a change of variable,
\begin{eqnarray*}
u_3 C_0\ell^{-2} \delta \omega^k &\leq& u_3 C_0\ell^{-1} \haus^1(E) \omega^k \\
&\leq& \int_{\R} \sum_j \#(\gamma_j^{-1}(t)) \dt
=\sum_j\int_0^{R_j} |\gamma_j'|(s)\ds \\
&=& \sum_j \int_0^{R_j}1 \ds \leq \haus^1(|\sigma^*|),
\end{eqnarray*}
where $C_0$, $u_3$, $\ell$, and $\delta$ do not depend on $k$. Thus \eqref{eq:lowhaus} holds. The proof is complete.


\def\cprime{$'$}

\end{document}